\documentclass[a4paper,10pt]{amsart}
\usepackage[plainpages=false]{hyperref}
\usepackage{amsfonts,amssymb,amsthm, mathrsfs, lscape}
\usepackage{verbatim}

\usepackage{latexsym}
\usepackage{amssymb}
\usepackage{graphics}
\usepackage{graphicx}
\usepackage[space]{cite}

\usepackage{latexsym}
\usepackage{amsfonts}
\usepackage{amsmath}
\usepackage{latexsym}
\usepackage{amsfonts}
\usepackage{amssymb}
\usepackage{rawfonts}

\usepackage[usenames,dvipsnames]{pstricks}
\usepackage{epsfig}
\usepackage{pst-grad} 
\usepackage{pst-plot} 
\usepackage[space]{grffile} 
\usepackage{etoolbox} 
\makeatletter 

\renewcommand{\Re}{{\operatorname{Re}\,}}

\renewcommand{\epsilon}{\varepsilon}

\newcommand{\kahler}{K\"ahler }
\newcommand{\KE}{K\"ahler-Einstein }

\newcommand{\R}{{\mathbb R}}
\newcommand{\C}{{\mathbb C}}
\newcommand{\Q}{{\mathbb Q}}
\newcommand{\Z}{{\mathbb Z}}

\newcommand{\gij}{{g_{i\bar{j}}}}
\newcommand{\tij}{{\vartheta_{i\bar{j}}}}
\newcommand{\tii}{{\vartheta_{i\bar{i}}(z',0)}}
\newcommand{\pij}{{\varphi_{i\bar{j}}}}

\renewcommand{\d}{\partial}
\newcommand{\dbar}{\bar\partial}
\newcommand{\ddbar}{\partial\dbar}

\newcommand{\Ric}{{\operatorname{Ric}}}

\renewcommand{\phi}{\varphi}

\newcommand{\gcal}{\mathcal{G}}
\newcommand{\hcal}{\mathcal{H}}

\newtheorem{theorem}{{Theorem}}[section]
\newtheorem{theo}{{Theorem}}[section]
\newtheorem{cor}[theorem]{{Corollary}}
\newtheorem{lem}[theorem]{{Lemma}}
\newtheorem{prop}[theorem]{{Proposition}}

\newenvironment{rem}{\medskip\noindent{\it Remark:\/} }{\medskip}

\theoremstyle{definition}

\numberwithin{equation}{section}

\def \C {\mathbb C}
\def \Z {\mathbb Z}
\def \R {\mathbb R}

\def \Q {\mathbb Q}

\def \xmd {X\backslash D}

\def \Exp {\text{Exp}}

\title[Bergman Kernels of the Cheng-Yau metrics]{Bergman Kernels of the Cheng-Yau metrics on quasi-projective manifolds}

\author{Jingzhou Sun}
\thanks{}
\address{Department of Mathematics, Shantou University, Shantou City, Guangdong Province 515063, China}

\begin{document}
	
	\begin{abstract}
We show the asymptotics of the Bergman kernel function near the smooth divisor at infinity of the Cheng-Yau metric on quasi-projective manifolds. In particular, we show that there is a quantum phenomenon for the points very close to the divisor at infinity.
	\end{abstract}

	\maketitle
	
	\tableofcontents

	\section{Introduction}

	Let $(M,\omega)$ be a \kahler manifold, $L\to M$ a line bundle with hermitian metric $h$. Then the Bergman space $\hcal_k$ consists of holomorphic sections of $L^k$ that are $L_2$-integrable, namely
	$$\int_M |s|_h^2\frac{\omega^n}{n!}<\infty.$$
	$\hcal_k$ is naturally a Hilbert space with the inner product defined as
	$$<s_1,s_2>=\int_M (s_1(z),s_2(z))_h\frac{\omega^n}{n!}.$$
	Then the density of states function or Bergman kernel function $\rho_k$, usually called Bergman kernel for short, is defined as
	$$\rho_k(z)=\sup_{s\in \hcal_k,\parallel s\parallel=1}|s(z)|^2.$$
	When $L$ is an ample line bundle over a projective manifold, the asymptotic of the Bergman kernel as $k\to\infty$ has been proved by Tian, Zelditch, Catlin, Lu\cite{Tian1990On, Zelditch2000Szego, Lu2000On, Catlin, MM}, and has been proved to be a very important tool in complex geometry, for example \cite{donaldson2001}\cite{Sun2011Expected}  \cite{Donaldson2014Gromov},	\cite{Donaldson15}. 
	
	When $M$ is not compact or when the metric $h$ on $L$ is not smooth, it becomes very hard, if not impossible, to describe the asymptotics of the Bergman kernel as nicely as in the compact smooth case. In \cite{AMM} and \cite{SS}, Auvray-Ma-Marinescu and Sun-Sun studied the Bergman kernels on punctured Riemann surfaces and showed that near each singularity the Bergman kernel is very close to the Bergman kernel of the standard punctured disk. In particular, it was first noticed in \cite{SS} that the quotient of the Bergman kernel near punctures with the Bergman kernel of the standard punctured disk is $k^{-\infty}$ close to 1. This was later generalized by Auvray-Ma-Marinescu in \cite{amm2} to larger neighborhoods of the punctures. In \cite{Zhou2024}, Zhou further generalized their result to the case of asymptotic complex hyperbolic cusps. There are also many other results in the literature studying the asymptotics of Bergman kernels of singular K\"ahler metrics, see for example  \cite{LL, RT, DLM}.

	Let $X$ be a smooth projective manifold of dimension $n$. Given a simple normal crossing divisor $D\subset X$ such that $K_X+[D]$ is ample, then it has been proved by Cheng-Yau Kobayashi, Tian-Yau and Bando(\cite{ChengYau2, Kobayashi,TianYau3,Bando}) that the quasi-projective manifold $\xmd$ admits a unique complete \KE metric $\omega_{KE}$, called Cheng-Yau metric, with finite volume and $\Ric(\omega_{KE})=-\omega_{KE}$. So $\omega_{KE}$ defines a hermitian metric on $K_X$ restricted to $\xmd$.
	
	In this article, we consider the Bergman space $\hcal_k$ consisting of holomorphic sections of $K_X^k$ on $\xmd$ that are $L_2$ integrable in the following sense:
	$$\int_{\xmd}|s|^2_{CY}\frac{\omega_{KE}^n}{n!}<\infty,$$
	where $|s|_{CY}$ is the metric defined by $\omega_{KE}$ on $K_X^k$. Assuming $D$ is smooth, we study the asymptotics of the Bergman kernels $\rho_k$ of $\hcal_k$ for the points very close to the divisor $D$.

	\

	Since $K_D=K_X+[D]$ is ample, by Aubin and Yau's famous theorem, there is a unique \KE metric $\omega_D$ satisfying $\Ric(\omega_{D})=-\omega_{D}$. In \cite{Schum98}, Schumacher proved that on the directions "parallel" to $D$, $\omega_{KE}$ converges to $\omega_D$. In the directions "orthogonal" to $D$, $\omega_{KE}$ looks like the standard Poincar\'{e} metric on the punctured disk.

Let $\{D_\alpha\}_{\alpha\in \Lambda}$ be the connected components of $D$.
We denote by $$B_{\alpha}=\frac{4(n-1)}{3}\frac{[D_\alpha]\cdot K_{D_\alpha}^{n-2}}{K_{D_\alpha}^{n-1}}.$$
We will adapt the notation $\epsilon(k)$ to denote a quantity that is asymptotically smaller than $\frac{1}{k^N}$ for all $N>0$. Since we will focus on points near $D$, for simplicity of notation, we will use $B$ in place of $B_{\alpha}$ around each ${D_\alpha}$. 
	Let $\tau(p_0)=-\log d^2(p_0,D)$, where $d(p_0,D)$ is the distance from $p_0$ to $D$. We first describe $\rho_{k+1}$ for points such that $\tau^{-1}=O(\frac{1}{\sqrt{k}\log k})$, which we will call "the inside". And our main result is the following theorem.
\begin{theo}\label{thm-main-inside}	
	Let $a$ be a positive integer and let $\Sigma_a$ and $\Sigma'_a$ be the hypersurfaces defined by $\tau=\frac{2k}{a}-\frac{B}{2}\log \frac{2k}{a}$ and $\tau=\frac{2k}{a}-\frac{B}{2}\log \frac{2k}{a}-\frac{k}{a(a+1)}$ respectively. 
	For $k$ large enough and for $a=O(\frac{\sqrt{k}}{\log k})$, we have
	$$\rho_{k+1}(p_0)=(1+O(\frac{1}{(\log k)^2}))\frac{2k^{3/2}}{\sqrt{\pi}a},$$
	for $p_0\in \Sigma_a$ and 
		$$\rho_{k+1}(p_0)=\epsilon(k),$$
	for $p_0\in \Sigma'_a$. We also have $$\rho_{k+1}(p_0)=\epsilon(k),$$
	when $\tau(p_0)>2k+\sqrt{k}\log k$.
\end{theo}
Notice that we have used $\rho_{k+1}$ instead of $\rho_{k}$ to be consistent with the other sections, where it is more convenient to use $k+1$.

\begin{rem}
	\begin{itemize}
		\item 	It is interesting to see that $\rho_{k+1}$ exhibit a quantum phenomenon as seen in the case of punctured Riemann surfaces \cite{Punctured}. We will refer to the region around the hypersurface $\Sigma_a$ as the $a$-th shell. 
		\item 	Actually, we have more precise estimates for $\rho_{k+1}$ if we combine lemma \ref{lem-concentrated} and propositions \ref{prop-upper-bound} and \ref{prop-lower-bound} than stated in this theorem, but we are not able to write down a summarization suitable for the introduction.
	\end{itemize}

\end{rem}

For points further away from $D$, we have the following result.

\begin{theo}\label{thm-main-neck}	
	For $k$ large enough, we have
	$$\rho_{k+1}(p_0)\leq \frac{k^{n-1}}{\pi^{n-1}}(4k+3\sqrt{k}\tau(p_0)),$$
	if $c_1\frac{\sqrt{k}}{\log k}<\tau(p_0)<c_2\sqrt{k}\log k$ for some fixed positive constants $c_1,c_2$.
\end{theo}
It would be very interesting to have a good lower bound for $\rho_{k+1}$ for points within this region. However, we are not able to do this at this moment. For points even further away from $D$, $\rho_{k+1}(p_0)$ should be close to the compact case, which we will not talk about in this article.

\

Let $L=K_X+D$. Let $s_D$ be a section of the line bundle $D$ whose zero divisor is $D$.
We consider $H^0(X,L^k-aD)$ as a subspace of $H^0(X,L^k)$, consisting of sections that vanish on $D$ of vanishing order $\geq a$. For each section $s\in \hcal_k$, we have a section $s\otimes s_D^{\otimes k}\in H^0(X\backslash D, kL)$. By the integrability condition, one can easily see that $s\otimes s_D^{\otimes k}$ extends to a holomorphic section on $X$ that vanishes on $D$. Conversely, for each section $f\in H^0(X,L^k-D)$, we have a holomorphic section $\frac{f}{s_D^{\otimes k}}$ of $K_X^k$ on $X\backslash D$, and it is easy to see that $\frac{f}{s_D^{\otimes k}}$ is integrable under the Cheng-Yau metric. Therefore $\hcal_k$ can be identified with $H^0(X,L^k-D)$ as a subspace of $H^0(X,L^k)$.

For $a>0$, we denote by the $j_a$ the embedding $H^0(X,L^k-aD)\to \hcal_k$ defined by $s\mapsto \frac{s\otimes s_D^{\otimes a}}{s_D^{\otimes k}}$. 
And we denote by $\hcal_{k,a}$ the image of $j_a$. We can get by induction an orthogonal decompostion
$$\hcal_k=H^0(X,L^k-bD)\oplus\oplus_{a=1}^{b-1}\gcal_{k,a}$$
where by $\gcal_{k,a}$ we mean a subspace of $\hcal_k$ that is orthogonal to $\hcal_{k,a+1}$. Clearly $\gcal_{k,a}$ is isomorphic to $H^0(D,L^k-aD)$. We denote by $\rho_{k,a}$ the Bergman kernel for $\hcal_{k,a}$ and by $\varrho_{k,a}$ the Bergman kernel for $\gcal_{k,a}$. So we have 
$$\rho_k=\sum_{a=1}^{b-1}\varrho_{k,a}+\rho_{k,b}.$$ 
We have the following theorem for $\rho_{k,a}$:
\begin{theo}\label{thm-main-hka}
    If $a=O(\frac{\sqrt{k}}{\log k})$, then we have 
    $$\rho_{k+1,a+1}=\epsilon(k)\rho_{k},$$
    for the points where $\tau>\frac{2k}{a}-\frac{B}{2}\log \frac{2k}{a}-\frac{k}{2a^2}$.
\end{theo}
Back to the shells of $\rho_k$, we expect that $\varrho_{k,a}$ only plays a role in the $a$-th shell, namely its contribution to the other shells is $\epsilon(k)$. If that is the case, then together with theorem \ref{thm-main-hka}, we can conclude that $\varrho_{k,a}$ dominates the $a$-th shell. However, we can only show this is indeed the case for small $a$, namely $a=O(1)$. The corresponding $a$-shells will be called inner shells. Our second main result is about the inner shells.

\begin{theo}\label{thm-main-vka}
    
    Let $a$ be a fixed integer. We have, for the points satisfying $\tau<\frac{2k}{a}-\frac{\sqrt{k}\log k}{a}$,
    $$\varrho_{k+1,a}=\epsilon(k).$$

\end{theo}

\

The structure of this article is as follows. We will first recall the asymptotic expansion of the Cheng-Yau metric around $D$ developed by Schumacher, Wu, Rochon-Zhang and Jiang-Shi etc. Then we estimates the Bergman kernel on disks vertical to $D$. Then we compare the Bergman kernel $\rho_k$ with the Bergman kernel on vertical disks by proving upper bounds in section \ref{sec-upper} and lower bounds in section \ref{sec-lower}. With these bounds, we then quickly prove theorems \ref{thm-main-inside}, \ref{thm-main-neck} and \ref{thm-main-hka}. In the end, by using extension theorems in the compact case, we will study the inner shells and prove theorem \ref{thm-main-vka}.

\textbf{Acknowledgements.} The author would like to thank Professor Yalong Shi for many helpful discussions.
\section{setting-up}
\subsection{General notions} In this article, we use the traditional notations $\Z,\Q,\R,\C$ for the sets of integers, rationals, real numbers and complex numbers respectively.

For a section $s$ of a hermitian line bundle, we use $|s|$ to denote the point-wise norm of $s$ and use $\parallel s\parallel$ to denote the $L_2$-norm. 

When we use Greek subscript $\alpha$ and $\beta$, we usually mean something related to a connected components of $D$.

For a local complex coordinates $z=(z_1,z_2,\cdots,z_n)$, we write $z'$ for the tuple $(z_1,z_2,\cdots,z_{n-1})$, so $z=(z',z_n)$. And to simplify the notations, we use the convention $$|dz|^2=(\sqrt{-1})^{n}dz_1\wedge d\bar{z}_1\wedge\cdots\wedge dz_{n}\wedge d\bar{z}_{n},$$
and also
$$|dz'|^2= (\sqrt{-1})^{n-1}dz_1\wedge d\bar{z}_1\cdots\wedge dz_{n-1}\wedge d\bar{z}_{n-1}.$$
\subsection{asymptotics of Cheng-Yau metric}

Given a hermitian metric $h$ on $L$ with positive curvature and any hermitian metric on $[D]$, the Carlson-Griffiths metric is defined as $$\omega_{CG}=\omega-\sqrt{-1}\ddbar\log (\log \frac{1}{\epsilon|s_D|^2})^2,$$
where $\omega=\sqrt{-1}\Theta_L$. 
When $\epsilon$ is small enough, $\omega_{CG}$ is a complete \kahler metric with finite volume on $\xmd$. By rescaling $s_D$ by $\sqrt{\epsilon}$, one can assume that $\epsilon=1$.
To find $\omega_{KE}$, one solves the complex Monge-Amp\'{e}re equation 
$$\log \frac{(\omega_{CG}+\sqrt{-1}\ddbar u)^n}{\omega_{CG}^n}-u=f,$$ 
where $f$ is smooth on $X\backslash D$ satisfying $\Ric(\omega_{CG})+\omega_{CG}=\sqrt{-1}\ddbar f$. Locally, if $e_D$ is a local frame of $[D]$, and $z=(z_1,\cdots,z_n)$ a local coordinates, then $e_L=e_D\otimes (dz_1\wedge\cdots\wedge dz_n)$ is a local frame of $L$. If $h$ is represented by $|e_L|^2_h=e^{-\psi}$ and the metric on $[D]$ by $|e_D|^2e^{-\phi}$, then there is a volume form $\gamma|dz|^2$ such that $e^{-\psi}=\frac{e^{-\phi}}{\gamma}$. Then 
$f$ can be chosen as $$\log \frac{\gamma|dz|^2}{|s_D|^2(\log |s_D|^2)^2\omega_{CG}^n}.$$ 
It has been proved that $\omega_{KE}$ is equivalent to $\omega_{CG}$.
The asymptotic expansion of $u$ in terms of $x=\frac{1}{-\log |s_D|^2 }$ was first obtained by Wu in \cite{WuDM}. Later Rochon and Zhang in \cite{RochonZhang} obtained a more precise asymptotic expansion and pointed out that there should in general be a $x\log x$ term. In \cite{Shi}, Jiang-Shi gave a different proof in the case when $D$ is smooth. We have basically followed the terminologies from \cite{Shi} for their simplicity.

\

Let $N_D$ be the normal bundle of $D$, with a hermitian metric defined by $\omega$. Let $N_D(R)\subset N_D$ be the disk bundle of radius $R$, and let $\Exp:N_D(R)\to X$ be the exponential map defined by the Riemannian metric associated to $\omega$. Then for $R$ small enough $\Exp$ is a diffeomorphism from $N_D(R)$ to its image. So the projection $N_D(R)\to D$ defines a projection $\pi_{\omega}$ from a neighborhood of $D$ in $X$ to $D$.

The main theorem in \cite{Shi} is
\begin{theo}
    Let $\omega_{KE}=\omega_{CG}+\sqrt{-1}\ddbar u$ be the unique complete \KE metric with finite volume on $X\backslash D$, and let $x=(-\log r^2)^{-1}$, where $r$ is the distance to $D$ with respect to some fixed \kahler metric on $X$. Then around $D$ we have a poly-homogeneous asymptotic expansion for $u$:
    $$u\sim \sum_{i\in I}\sum_{j=0}^{N_i}c_{i,j}x^i(\log x)^j,$$
    where $I$ is the index set determined by the eigenvalues of the Laplacian operator of the unique \KE metric on $D$ and $c_{i,j}$'s  are smooth functions on $D$, regarded as functions in a neighborhood of $D$ via the projection $\pi_{\omega}$. The precise meaning of the above expansion is  that 
    $$u-\sum_{i\in I, i\leq k}\sum_{j=0}^{N_i}c_{i,j}x^i(\log x)^j=O(x^{k+}),$$
    where $O(x^{k+})$ is the next term of $k$ in $I$. 
\end{theo}

Let $\sigma=-\log |s_D|^2$, they also showed that $x$ can be replaced by $\frac{1}{\sigma}$. So we will use $x=\frac{1}{\sigma}$.

In \cite{RochonZhang}, Rochon-Zhang showed that the pairs $(i,j)$ that appear in the expansion satisfying $i\leq 1$ are $(0,0),(1,1)$ and $(1,0)$.
They also showed that $c_{1,1}=\frac{4(n-1)}{3}\frac{[D_\alpha]\cdot K_{D_\alpha}^{n-2}}{K_{D_\alpha}^{n-1}}$ is a constant around each ${D_\alpha}$. If the restriction of $\omega$ to $D$ is the \KE metric on $D$, then in \cite{Shi}, they showed that $c_{0,0}$ is a constant around each ${D_\alpha}$ and $f=-c_{0,0}+O(x)$. 
So under this assumption, we have 
$$\frac{\omega_{KE}^n}{\omega_{CG}^n}=e^{u+f}=O(\frac{\log \sigma}{\sigma}).$$
Since $$e^f=\frac{\omega^n}{|s_D|^2(\log |s_D|^2)^2\omega_{CG}^n},$$
we have $$\omega_{KE}^n=e^u\frac{\gamma|dz|^2}{|s_D|^2(\log |s_D|^2)^2}.$$
Clearly, adding a constant to $f$ while substracting the same constant from $u$ does not affect the Monge-Amp\'{e}re equation. So when the restriction of $\omega$ to $D$ is $\omega_D$ we can simply make $c_{0,0}=0$.

Around each point $p\in D$, we can find local coordinates $(z_1,\cdots,z_n)$ centered at $p$ such that $D$ is defined by $z_n=0$. So $z'=(z_1,\cdots,z_{n-1})$ are local coordinates of $D$.
For any local frame $e_D$ of $[D]$, $s_D=z_nG(z)e_D$ for some holomorphic function $G(z)$ that is non-vanishing on $D$. So $|s_D|^2=|z_n|^2|G|^2e^{-\phi}$. Therefore
\begin{equation}\label{eqn-wke}
\omega_{KE}^n=e^u\frac{\gamma|dz|^2}{|z_n|^2|G|^2e^{-\phi}\sigma^2}=
e^u\frac{e^{\psi}/|G|^2}{|z_n|^2\sigma^2}|dz|^2
\end{equation}

We can define $ds_D=G(z)dz_n\otimes e_D$ which is globally defined as a section of $N_D^*\otimes [D]$, where $N_D^*$ is the conormal bundle of $D$. Then we recall that the identification of $L|_D$ with $K_D$ is given by $s\mapsto \frac{s}{ds_D}$ for any local section $s$ of $L$, which in local coordinates is 
$$Fe_D\otimes dz_1\wedge\cdots\wedge dz_n\mapsto \frac{F}{G}dz_1\wedge\cdots\wedge dz_{n-1}. $$
Under this identification, each hermitian metric $h$ on $L$ defines a hermitian metric on $K_D$. 

Let $\omega_D$ be the \KE metric on $D$ satisfying $\Ric(\omega_D)=-\omega_D$. Then $|dz_1\wedge\cdots\wedge dz_{n-1}|^2_{\omega_D}=\frac{(n-1)!|dz'|^2}{\omega_D^{n-1}}$ defines a hermitian metric on $K_D$,
hence on $L|_D$ under the identification. By theorem 4 in \cite{Schum98}, we can extend this metric to a hermitian metric $h$ on $L$ whose curvature is $-i\omega$, where $\omega\in c_1(L)$ is a \kahler form. So we have $j^*\omega=\omega_D$, where $j$ is the embedding of $D\to X$. 

\

From now on, we will fix such an $h$ and $\omega$.

\

If we write $\omega=\sqrt{-1}\sum \tij dz_i\wedge\d\bar{z}_j$, we define $H=\det \{\tij \}_{i<n,j<n}$.
Then we have $H(z',0)=\frac{\omega_D^{n-1}}{(n-1)!|dz'|^2}$. If $h$ is represented by $|e_L|^2_h=e^{-\psi}$, then since $h$ extends the metric on $K_D$ defined by $\frac{(n-1)!|dz'|^2}{\omega_D^{n-1}}$, we get that 
\begin{equation}\label{eqn-H}
H(z',0)=\frac{e^{\psi(z',0)}}{|G(z',0)|^2}.
\end{equation}
In local coordinates, if we write $\omega_{CG}=\sqrt{-1}\sum g_{i\bar{j}}dz_i\wedge\d\bar{z}_j$, then 
$\gij=\tij-2x\pij+\frac{2x^2}{r^2}V_i\bar{V}_j$, where $V_i=r\phi_i-\delta_{in}e^{-\sqrt{-1}\theta}$.  Then we have
$$\omega_{CG}^n=(1-2x\Delta_D\phi+O(x^2))\frac{2n!x^2}{r^2}H(\sqrt{-1})^{n}dz_1\wedge d\bar{z}_1\wedge\cdots\wedge dz_{n}\wedge d\bar{z}_{n}.$$
Let $\omega_0=\sqrt{-1}\sum g'_{i\bar{j}}dz_i\wedge\d\bar{z}_j$, where $g'_{i\bar{j}}$ satsifies
\begin{itemize}
    \item $g'_{i\bar{j}}(z)=g_{i\bar{j}}(z',0)$ for $i<n, j<n$;
    \item $g'_{n\bar{n}}=\frac{2}{r^2(\log r^2)^2}$;
    \item $g'_{i\bar{n}}=0$ for $i<n$.
\end{itemize}
Then for any tangent vector $v=\sum a_i\frac{\partial}{\partial z_i}$, we have 
$$|v|^2_{\omega_{CG}}=(1+O(x))|v|^2_{\omega_{0}}.$$
It is easy to estimate the injectivity radius under the metric $\omega_0$. There are positive constants $\lambda_1,\lambda_2$ such that the injectivity radius $\nu_0(z)$ satsifies $$\lambda_1\frac{1}{-\log r^2}<\nu_0(z)<\lambda_2\frac{1}{-\log r^2}. $$ 
So we can estimate the injectivity radius of the metric $\omega_{CG}$, and hence that of $\omega_{KE}$. Therefore there are positive constants $\lambda'_1,\lambda'_2$ such that the injectivity radius $\nu_{KE}(z)$ under the Riemannian metric correspongding to $\omega_{KE}$ satsifies $$\lambda'_1\frac{1}{-\log r^2}<\nu_{KE}(z)<\lambda'_2\frac{1}{-\log r^2}. $$ 

\

\section{fiber-wise calculations}

Denote by $\phi_1=\phi-\log |G(z',0)|^2$.
For fixed $z'$, let $t=-\log |z_n|^2$, then $\sigma=\phi_1 +t$. Then since
$\frac{1}{\sigma}=\frac{1}{t}(1+\sum_{i=1}^{\infty}(-\frac{\phi_1}{t}))^i$, the expansion of $u$ in terms of $\frac{1}{\sigma}$ can be rewritten as expansion in terms of $\frac{(\log t)^j}{t^i}$. Therefore, since $|z_n|=O(x^\infty)$, by formula \ref{eqn-H},
\begin{equation}
\omega_{KE}^{n}=n!\frac{H(z',0)}{Yt^2|z_n|^2}|dz|^2,
\end{equation}
where $$Y=1+c_{1,1}\frac{\log t}{t}+\frac{1}{t}(2\phi_1(z',0)-c_{1,0}(z',0))+\sum_{i\in I'}\sum_{j=0}^{N'_i}b_{i,j}\frac{(\log t)^j}{t^i},$$ where, since the eigenvalues of the Laplacian of $\omega_D$ is discrete, $I'$ is an discrete index set with $\min\{i|i\in I'\}>1$.
Let $e_K=\frac{1}{z_n}dz_1\wedge\cdots\wedge dz_{n}$ be a local frame of $K_X$ on $\xmd$, then the Cheng-Yau metric is $h_{CY}=|e_K|^2_{CY}=Y\frac{t^2}{H(z',0)}$.

For a fixed $z'$, we write $A=2\phi_1(z',0)-c_{1,0}(z',0)$, $B=c_{1,1}$, $C=\frac{1}{H(z',0)}$, then $h_{CY}=Ct^2(1+\frac{B\log t}{t}+\frac{A}{t}+\sum_{i\in I'}\sum_{j=0}^{N'_i}b_{i,j}\frac{(\log t)^j}{t^i})$.

For any holomorphic section $s\in \hcal_{k+1}$, locally $s=F(z)(dz_1\wedge\cdots\wedge dz_n)^{\otimes (k+1)}$. Then $$|s|_{CY}^2=|F|^2h_{CY}^{k+1}.$$
So $|s|_{CY}^2\frac{\omega_{KE}^n}{n!}=|F|^2h_{CY}^{k}|dz|^2$. To calculate the integral of $|s|_{CY}^2\frac{\omega_{KE}^n}{n!}$, we look at the integral of $|F|^2h_{CY}^{k}$ on $z_n$-disks.


Let $h_0(t)=t^2(1+\frac{B\log t}{t}+\frac{A}{t}+\sum_{i\in I', i\leq 6}\sum_{j=0}^{N'_i}b_{i,j}(z',0)\frac{(\log t)^j}{t^i})$, then for $t>k^{1/3}$, $h_{CY}^{k}=C^kh_0(t)^{k}(1+O(\frac{k}{t^{6+\epsilon}}))$ for some $\epsilon>0$.

Let $\mathbb{B}_3\subset \C$, with coordinate $w$, be the ball defined by $t=-\log |w|^2>k^{1/3}$. Let $\alpha=\frac{\sqrt{-1}dw\wedge d\bar{w}}{t^2|w|^2 \rho_0}$
We study the Bergman kernel of $(\mathbb{B}_3,h_0^{k+1},\alpha)$. So the norms of the functions $w^i$ are $\parallel w_i\parallel^2=2\pi J_i,$ where
$$J_i=\int_0^{k^{1/3}}h_0^k(t)e^{-it}dt.$$
Then the Bergman kernel is 
$$\mathfrak{B}_{k+1}=h_0^{k+1}(t)\sum_{i=1}^\infty\frac{|w|^{2i}}{2\pi J_i}. $$
And we can define $$\mathfrak{B}_{k+1,a}=h_0^{k+1}(t)\sum_{i\geq a}^\infty\frac{|w|^{2i}}{2\pi J_i},$$
for $a\geq 1$.

The ideas for studying $\mathfrak{B}_{k+1}$ are very close to those developed from \cite{SunSun} to \cite{Punctured}. Since the difference of the settings are not negligible, we have to include a proof here.
We first estimate $J_a$ for $a=O(\sqrt{k}\log k)$.

We have $$J_a=\int_0^{k^{1/3}}e^{g_a(t)}dt,$$
where $g_a(t)=k\log h_0(t)-at=2k\log t+kB\frac{\log t}{t}+k\frac{A}{t}+\sum_{i\in I''}\sum_{j=0}^{N''_i}kb'_{i,j}\frac{(\log t)^j}{t^i} $, with $\min\{i|i\in I''\}>1$. We denote by $\kappa(t)=\sum_{i\in I''}\sum_{j=0}^{N''_i}b'_{i,j}\frac{(\log t)^j}{t^i}$. Then 
$g'_a(t)=\frac{2k}{t}+k\frac{B-A}{t^2}-kB\frac{\log t}{t^2}-a+k\kappa'(t)$ and $g''_a(t)=-\frac{2k}{t^2}+k\frac{2A-3B}{t^3}+2kB\frac{\log t}{t^3}+k\kappa''(t)$.  Since $\kappa''(t)=o(\frac{1}{t^3})$, we have that 
$$g''_a(t)<0,$$
for $k$ large enough. So $g_a(t)$ is a strictly concave function which attains its only maximum at $t_a$ satisfying
$$\frac{2k}{t_a}+k\frac{B-A}{t_a^2}-kB\frac{\log t_a}{t_a^2}+k\kappa'(t_a)=a.$$
It is easy to see that $t_a=\frac{2k}{a}+O(\log k)$.
Since $g^{(3)}_a(t)=-\frac{4k}{t^3}-6kB\frac{\log t}{t^4}+k\frac{11B-6A}{t^4}+k\kappa^{(3)}(t)$, we have
$g''_a(t)=g''_a(\frac{2k}{a})+O(\frac{a^3\log k}{k^2})=-\frac{a^2}{2k}+O(\frac{a^3\log k}{k^2})$, for $t=\frac{2k}{a}+O(\log k)$. Therefore, we have 
\begin{eqnarray*}
    (t_a-\frac{2k}{a})(-\frac{a^2}{2k}+O(\frac{a^3\log k}{k^2}))&=&g_a'(t_a)-g'_a(\frac{2k}{a})\\&=&Ba^2\frac{\log \frac{2k}{a}}{4k}-a^2\frac{B-A}{4k}-k\kappa'(\frac{2k}{a})
\end{eqnarray*}
Therefore
$$t_a=\frac{2k}{a}-\frac{B}{2}\log\frac{2k}{a}+\frac{B-A}{2}+\frac{2k^2}{a^2}\kappa'(\frac{2k}{a})+O(\frac{a(\log k)^2}{k}).$$
If we let $t'_a=\frac{2k}{a}-\frac{B}{2}\log\frac{2k}{a}+\frac{B-A}{2}+\frac{2k^2}{a^2}\kappa'(\frac{2k}{a})$ and replace $\frac{2k}{a}$ by $t_a'$ to approximate $t_a$, we get that $t_a=t_a'+O(\frac{a(\log k)^2}{k})$. And we have
$g''_a(t)=g''_a(t_a')+O(\frac{a^4(\log k)^2}{k^3})=-\frac{a^2}{2k}+O(\frac{a^3\log k}{k^2})$, for $t=t_a'+O(\frac{a(\log k)^2}{k})$. Therefore, assuming $a=O(\sqrt{k}\log k)$, one can see that
$$t_a=\frac{2k}{a}-\frac{B}{2}\log\frac{2k}{a}+\frac{B-A}{2}+\mathfrak{K} (\frac{2k}{a})+O(\frac{(\log k)^2}{k^{1+\epsilon}}),$$
for some $\epsilon>0$, where $\mathfrak{K}(t)=\sum_{i\in \mathfrak{I} }\sum_{j=0}^{M_j}d_{i,j}\frac{(\log t)^j}{t^i}$ for some discrete index set $\mathfrak{I}$ satisfying $\min\{i\in \mathfrak{I}\}>0$ and $\max\{i\in \mathfrak{I}\}\leq 2$, and moreover, the coefficients $d_{i,j}$ are polynomials of $A,B$ and the coefficients of $\kappa$. So we have that for  $a=O(\sqrt{k}\log k)$, 
$$h_0(t_a)=(\frac{2k}{a})^2(1+\frac{Ba}{2k}+\mathfrak{K}'(\frac{2k}{a})+O(\frac{a(\log k)^2}{k^{2+\epsilon}})),$$ 
where the function $\mathfrak{K}'$ satisfies the same conditions as $\mathfrak{K}$.

The Gaussian integral $\int_{-\infty}^{\infty}e^{-x^2/2}dx=\sqrt{2\pi}$. The mass is concentrated within the region of radius $C\log k$, for a fixed $C>0$, in the sense that 
$$\int_{|x|>C\log k}e^{-x^2/2}dx=\epsilon(k).$$
And since $g''_a(t)=-\frac{a^2}{2k}(1+O(\frac{\log k}{\sqrt{k}}))$ for $|t-t_a|=O(\frac{\sqrt{k}\log k}{a})$ we have the following
\begin{lem}\label{lem-ja}
    For a fixed $C>0$, we have 
    $$J_a=(1+\epsilon(k))\int_{|t-t_a|<\frac{C\sqrt{k}\log k}{a}}e^{g_a(t)}dt.$$
    Therefore, we have
    $$J_a=h_0^k(t_a)e^{-at_a}\sqrt{\pi k}\frac{2}{a}(1+O(\frac{\log k}{\sqrt{k}}))$$
\end{lem}

We can write $\frac{|w|^{2i}}{J_i}$ as $e^{-it-\log J_i}$. Let $\iota(i)=-it-\log J_i$ be the exponent. Clearly, $\iota(i)$ is defined for all real $i\geq 1$. We claim that $\iota(i)$ is a concave function of $i$. We can calculate
\begin{eqnarray*}
    \iota''(i)&=&-\frac{d^2}{di^2}\log J_i\\
    &=&\frac{1}{J^2_i}[(\int th_0^k(t)e^{-it}dt)^2-\int h_0^k(t)e^{-it}dt\int t^2h_0^k(t)e^{-it}dt]\leq 0,
\end{eqnarray*}
by Cauchy-Schwarz inequality. We denote by $\Lambda_i(t)=h_0^k(t)e^{-it-\log J_i}$. Then $$ \mathfrak{B}_{k+1}(t)=h_0(t)\sum \Lambda_i(t),$$and we have $$\Lambda_a(t_a)=\frac{a}{2\sqrt{\pi k}}(1+O(\frac{\log k}{\sqrt{k}})).$$
And 
\begin{eqnarray*}
    \Lambda_{a+1}(t_a)&=&\Lambda_{a+1}(t_{a+1})e^{-\frac{a^2}{4k}(1+O(\frac{a\log k}{k}))(\frac{2k}{a^2})^2}\\
    &=&\frac{a+1}{2\sqrt{\pi k}}(1+O(\frac{\log k}{\sqrt{k}}))e^{-\frac{k}{a^2}(1+O(\frac{a\log k}{k}))},
\end{eqnarray*}
since $t_{a+1}-t_a=-\frac{2k}{a^2}(1+O(\frac{a\log k}{k}))$.
So when $a=O(\frac{\sqrt{k}}{\log k})$, we have 
$\Lambda_{a+1}(t_a)=\epsilon(k)\Lambda_{a}(t_{a})$, and similarly, $\Lambda_{a-1}(t_a)=\epsilon(k)\Lambda_{a}(t_{a})$. Therefore, we have 
\begin{lem}\label{lem-inside}
    For $a=O(\frac{\sqrt{k}}{\log k})$, we have
    \begin{eqnarray*}
    \mathfrak{B}_{k+1}(t_a)&=&(1+\epsilon(k))\Lambda_a(t_a)h_0(t_a) \\
    &=&(1+O(\frac{\log k}{\sqrt{k}}))\frac{2k^{3/2}}{\sqrt{\pi}a}.
    \end{eqnarray*} 
    
\end{lem}
We can say more in this direction. Let $t=t_a-\frac{k}{2a^2}$, then $t_{a+1}-t=-\frac{3k}{2a^2}(1+O(\frac{a\log k}{k}))$, and $t_{a}-t=\frac{k}{2a^2}(1+O(\frac{a\log k}{k}))$. So $\Lambda_{a+1}(t)=\epsilon(k)\Lambda_{a}(t)$. Therefore, by monotonicity, we have
\begin{lem}\label{lem-concentrated}
    For $a=O(\frac{\sqrt{k}}{\log k})$, we have 
    $$\sum_{i\geq a+1}\Lambda_i(t)=\epsilon(k)\Lambda_a(t), $$
    for $t\geq t_a-\frac{k}{2a^2}$.
    And by symmetry, $$\sum_{i\leq a-1}\Lambda_i(t)=\epsilon(k)\Lambda_a(t), $$
    for $t\leq t_a+\frac{k}{2a^2}$.
\end{lem}
One can notice that $t_a$ is not the maximum point of the function $\Lambda_a(t)h_0(t) $. Let $t_a'$ be the maximum point, then we can easily get that $t_a'-t_a\approx \frac{2}{a}$. And so $$\Lambda_a(t_a')h_0(t_a')=(1+O(\frac{1}{k}))\Lambda_a(t_a)h_0(t_a).$$
So we can use $t_a$ instead of $t_a'$ for simplicity. Moreover, for $|t-t_a|=O(1)$, we have $$\mathfrak{B}_{k+1}(t)=(1+O(\frac{1}{(\log k)^2}))\frac{2k^{3/2}}{\sqrt{\pi}a}.$$
When $a=1$, we have $\mathfrak{B}_{k+1}(t)=(1+\epsilon(k))\Lambda_1(t)$ for $$t\geq t_1-\frac{k}{2}.$$
Then since $\Lambda_1(t)=\epsilon(k)$ for $t>t_1+\frac{1}{2}\sqrt{k}\log k$. We get that 
\begin{equation}\label{for-most-in}
    \mathfrak{B}_{k+1}(t)=\epsilon(k),\quad \text{for } t>t_1+\frac{1}{2}\sqrt{k}\log k.
\end{equation}

To estimate $\mathfrak{B}_{k+1}(t_a)$ for $a>\frac{\sqrt{k}}{\log k}$, we use
the concavity of $\iota(a)$. We need to estimate $\iota''(a)$. 
Let $d\mu_a=h_0^ke^{-at}\frac{dt}{J_a}$ be the probability measure. Let $\tau_a=\int td\mu_a$, then we have
\begin{eqnarray*}
    \iota''(a)&=&(\tau_a)^2-\int t^2d\mu_a\\
    &=&-\int (t-\tau_a)^2d\mu_a
\end{eqnarray*}
So we have 
$$|\iota''(a)|\leq \int (t-t_a)^2h_0^k(t)e^{-at}d\mu_a $$
$$|\iota''(a)|>\min \{\int_{t\geq t_a} (t-t_a)^2d\mu_a,\int_{t\leq t_a} (t-t_a)^2d\mu_a \}. $$
Since $\int_{-\infty}^{\infty}x^2e^{-x^2/2}dx=\sqrt{2\pi}$, we have 
$\int (t-t_a)^2h_0^k(t)e^{-at}d\mu_a=\frac{2k}{a^2}(1+O(\frac{\log k}{\sqrt{k}}))$. And so $\int_{t\geq t_a} (t-t_a)^2d\mu_a=\frac{k}{a^2}(1+O(\frac{\log k}{\sqrt{k}}))$ and $\int_{t\leq t_a} (t-t_a)^2d\mu_a=\frac{k}{a^2}(1+O(\frac{\log k}{\sqrt{k}}))$. Therefore we get
\begin{equation}
-\frac{2k}{a^2}(1+O(\frac{\log k}{\sqrt{k}})) \leq \iota''(a)<-\frac{k}{a^2}(1+O(\frac{\log k}{\sqrt{k}}))
\end{equation}
Then we can prove
\begin{lem}\label{lem-bk-between}
    For $c_1\frac{\sqrt{k}}{\log k}<a<c_2\sqrt{k}\log k$, we have 
    $$\mathfrak{B}_{k+1}(t_a)<4k+\frac{2k\sqrt{k}}{a}.$$
    And when $a$ is an integer, we also have 
    $$\frac{k\sqrt{k}}{a}<\mathfrak{B}_{k+1}(t_a)$$
\end{lem}
\begin{proof}
    For $a>\frac{\sqrt{k}}{\log k}$, we consider $\Gamma_{a+\frac{a\log k}{\sqrt{k}}}(t_a)$. We have $$t_{a+\frac{a\log k}{\sqrt{k}}}-t_a=-\frac{2k}{a^2}\frac{a\log k}{\sqrt{k}}(1+O(\frac{a\log k}{k}+\frac{\log k}{\sqrt{k}})).$$ Therefore, $$\Gamma_{a+\frac{a\log k}{\sqrt{k}}}(t_a)=\epsilon(k).$$
    Similarly, $$\Gamma_{a-\frac{a\log k}{\sqrt{k}}}(t_a)=\epsilon(k).$$
    This means that $\Gamma_i(t_a)$ attains its maximum for some $a-\frac{a\log k}{\sqrt{k}}<i<a+\frac{a\log k}{\sqrt{k}}$. We can use the inequality
    $$1<\sum_{i\in \Z}e^{-\frac{\lambda}{2}i^2}<\sqrt{\frac{2\pi}{\lambda}}+2$$
    to get $$\frac{a}{2\sqrt{\pi k}}(1+O(\frac{\log k}{\sqrt{k}}))\geq \sum \Lambda_i(t_a)$$ and 
    \begin{eqnarray*}
        \sum \Lambda_i(t_a)&<&\frac{a}{2\sqrt{\pi k}}(1+O(\frac{\log k}{\sqrt{k}}))(a\sqrt{\frac{2\pi}{k}}(1+O(\frac{\log k}{\sqrt{k}})+2)\\
        &=&\frac{a}{2\sqrt{\pi k}}(1+O(\frac{\log k}{\sqrt{k}}))(a\sqrt{\frac{2\pi}{k}}+2)\\
        &=&(\frac{a^2}{\sqrt{2}k}+\frac{a}{2\sqrt{\pi k}})(1+O(\frac{\log k}{\sqrt{k}}))
    \end{eqnarray*}
    Since $h_0(t_a)=(\frac{2k}{a}^2(1+O(\frac{a\log k}{k})))$, we get that when $a$ is an integer,
    $\mathfrak{B}_{k+1}(t_a)>\frac{k\sqrt{k}}{a}$  and
    $\mathfrak{B}_{k+1}(t_a)<4k+\frac{2k\sqrt{k}}{a}$.
    When $0<\eta<1$, we do not have a good lower bound, but we still have 
    $$\sum_{i\in \Z}e^{-\frac{\lambda}{2}(i+\eta)^2}<\sqrt{\frac{2\pi}{\lambda}}+2.$$
    So we get $\mathfrak{B}_{k+1}(t_a)<4k+\frac{2k\sqrt{k}}{a}$ for general $a$.
\end{proof}
Let $a=\sqrt{k}\log k$, then since $\frac{(2a)^2}{4k}*a^2=k(\log k)^4$, we get $$\Gamma_{2a}(t_a)<3e^{-\frac{1}{4}k(\log k)^4}\Gamma_{a}(t_a).$$ So by the concavity of $\iota$, we get the following
\begin{lem}\label{lem-2a}
    For $t<t_{\sqrt{k}\log k}$, we have $$h_0^{k+1}(t)\sum_{i=2\sqrt{k}\log k}^\infty\frac{|w|^{2i}}{2\pi J_i}<4e^{-\frac{1}{4}k(\log k)^4}\mathfrak{B}_{k+1}(t)$$
\end{lem}

\section{upper bound}\label{sec-upper}
We can choose local coordinates $z=(z_1,\cdots,z_n)$ centered at $p\in D$ such that $\tij=\delta_{ij}+O(|z|)$. Then by a unitary transformation, we can make $\frac{\partial}{\partial z_n}$ orthogonal to the tangent space $T_pD$ of $D$ at $p$. So $\frac{\partial}{\partial z_i}\in T_pD$ for $i<n$. Then if $D$ locally defined by a holomorphic function $F=0$, we have $\frac{\partial}{\partial z_n}F(0)\neq 0$ and $\frac{\partial}{\partial z_i}F(0)=0$ for $i<n$. So after a rescaling the vanishing order of $F_1(z)=F-z_n$ is $\geq 2$. So if we let $w_n=F(z)$, then under the new coordinates $z=(z',w_n)$, we still have $$\tij=\delta_{ij}+O(|z|).$$
Then we can choose a new coordinates $w_i=z_i+f_i(z'), i<n$ such that under the new coordinates $(w',w_n)=(w_1,\cdots,w_n)$, $\tij=\delta_{ij}+O(|w'|^2+|w_n|)$ for $i<n$ and $j<n$. By abuse of notation, we still use $z=(z',z_n)=(w',w_n)$. Notice that we still have that $D$ is defined by $z_n=0$. So now $H(z',0)=1+O(|z|^2)$. Since $\ddbar \log H(z',0)=\omega_D$, we have $\log H(z',0)=|z'|^2+2\Re(\sum P_{ab}z_az_b+\sum Q_{abc}z_az_bz_c)+O(|z'|^4)$. Since  $\tij(z',0)=\delta_{ij}+O(|z'|^2)$, one sees that the $(2,0)$ terms and the $(3,0)$ terms of $\log H(z',0)$ only come from the $(2,0)$ terms and the $(3,0)$ terms of $\tii$, which in turn only come from the $(3,1)$ terms of the form $\sum V_{abd}z_az_bz_d\bar{z}_d$ and the $(4,1)$ terms of the form $\sum U_{abcd}z_az_bz_cz_d\bar{z}_d$ of $\log H(z',0)$. Then we change coordinates $w_d=z_d+\sum V_{abd}z_az_bz_d+\sum U_{abcd}z_az_bz_cz_d$ and the in the new coordinates, $\log H(z',0)$ no longer has those $(3,1)$ terms and $(4,1)$ terms with the special forms. Then the Taylor expansion of $\tii$ has no $(2,0)$ terms and $(3,0)$ terms for each $i<n$. Therefore, we now have $$\log H(z',0)=|z'|^2+O(|z'|^4).$$
In remark 6.2 of \cite{Shi}, they commented that we can rescale the metric on $[D]$ by $e^F$ with some function $F$ to make $c_{1,0}$ a constant without changing the metric on $L$. In this section, we will use this choice, namely $c_{1,0}$ is now a constant. 

Recall that $\phi_1(z',0)=\phi(z',0)-\log |G(z',0)|^2$ and $|s_D|^2=|z_n|^2e^{-\phi_1}$. We also want to make the Taylor expansion of $\phi_1(z',0)$ at $z'=0$ have no degree 1 terms. Let $e^{-\phi_1(z',0)}=\lambda(1+\Re \sum P_az_a+O(|z'|^2))$. Then we let $w_n=z_n+\sum \sum P_az_az_n$. So with this new coordinates $z=(z',z_n)=(z',w_n)$, we have $e^{-\phi_1(z',0)}=\lambda(1+O(|z'|^2))$, while having $D$ being defined by $z_n=0$.
Recall that $ds_D$ is a non-vanishing global section of $N_D^*\otimes [D]$. If we equip the line bundle $N_D^*\otimes [D]$ with the hermitian metric induced from $\omega$ and the chosen metric $[D]$, we get that $$|ds_D(p)|^2=|G(p)|^2e^{-\phi(p)}.$$
So although $\phi_1(z',0)$ is only defined locally, $\phi_1(0,0)=-\log |ds_D(p)|^2$ is independent of the choice of coordinates and local frame and hence being uniformly bounded on $D$. So now all the coefficients in the asymptotic of $h_0(t)$ at $z'=0$ depend only on $p$, namely they are independent the choice of coordinates and local frame. So if we denote the Bergman kernel $\mathfrak{B}_{k+1}(t)$ at $z'$ by $\mathfrak{B}_{k+1}(z',t)$, then $\mathfrak{B}_{k+1}(0,t)$ can be denoted by $$\mathfrak{B}_{k+1}(p,t).$$

For any holomorphic section $s\in \hcal_{k+1}$ of unit norm, locally $s=F(z)(dz_1\wedge\cdots\wedge dz_n)^{\otimes (k+1)}$. Then $|s|_{CY}^2\frac{\omega_{KE}^n}{n!}=|F|^2h_{CY}^{k}\frac{|dz|^2}{|z_n|^2}$. Let $U_k=\{(z',z_n)\in \C^n||z'|<\sqrt{k}\log k,t<k^{1/3}  \}$. 
We want to calculate the integral of $|s|_{CY}^2\frac{\omega_{KE}^n}{n!}$ on $U_k$. Since
$|F|^2h_{CY}^{k}=(H(z',0))^{-k}|F|^2(1+O(\frac{1}{t^{6+\epsilon}}))h_0^k$, we only need to calculate the integral $$\int_{U_k}H^{-k}(z',0)|F|^2h_0^k\frac{|dz|^2}{|z_n|^2}.$$
One may need to be reminded that $h_0(t)$ depends on $z'$, so we include this dependence in the notation $h_0(z',t)$.

In local coordinates, write $F(z)=\sum_{i=1}^{\infty} F_i(z')z_n^i$ as a power series of $z_n$.  
Then for fixed $z'$,  $$\int_{t<k^{1/3}}H^{-k}|F|^2h_0^k\frac{|dz_n|^2}{|z_n|^2}=H^{-k}\sum_{i=1}^{\infty}|F_i|^2|\int_{t<k^{1/3}}|z_n^i|^2h_0^k\frac{|dz_n|^2}{|z_n|^2}.$$

We deal with the part $\sum_{i\geq 2\sqrt{k}\log k}$ first. Recall that $A=2\phi_1(z',0)-c_{1,0}(z',0)$. By our choice of metric and coordinates, we have $A(z')=A(0)+O(|z'|^2)$. Therefore $h_0(z',t)=h_0(0,t)(1+O(\frac{|z'|}{t^{1+\epsilon}}))$, for some $\epsilon>0$. So $h_0^k(z',t)=h^k_0(0,t)e^{O(k^{1/6-\epsilon'})}$. 
We denote by $$P_1=\sum_{i\geq 2\sqrt{k}\log k}|F_i(0)|^2\int_{t<k^{1/3}}|z_n^i|^2h_0^k(0,t)\frac{|dz_n|^2}{|z_n|^2},$$
and $$Q_1=\int_{|z'|<\frac{\log k}{\sqrt{k}}}H^{-k}(z',0)\sum_{i\geq 2\sqrt{k}\log k}|F_i|^2\int_{t<k^{1/3}}|z_n^i|^2h_0^k\frac{|dz_n|^2}{|z_n|^2}.$$
Then, since $s$ is of unit norm, $Q_1\leq (1+o(k^{-2}))$. And since
\begin{eqnarray*}
    \int_{|z'|<\frac{\log k}{\sqrt{k}}}|F_i(z')|^2H^{-k}|dz'|^n&\geq&(1+O(\frac{(\log k)^2}{k})) \int_{|z'|<\frac{\log k}{\sqrt{k}}}|F_i(z')|^2e^{-k|z'|^2}\\
    &\geq&(1+O(\frac{(\log k)^2}{k}))\frac{\pi^{n-1}}{k^{n-1}}|F_i(0)|^2
\end{eqnarray*}
we have 
$$P_1\leq \frac{k^{n-1}}{\pi^{n-1}}e^{o(k^{1/6})}Q_1=e^{o(k^{1/6})}.$$
So \begin{eqnarray*}
    |\sum_{i\geq 2\sqrt{k}\log k}F_i(0)z_n^i|^2&\leq &P_1\sum_{i\geq 2\sqrt{k}\log k}\frac{|z_n^i|^2}{J_i}\\
    &=&e^{o(k^{1/6})}\sum_{i\geq 2\sqrt{k}\log k}\frac{|z_n^i|^2}{J_i}
\end{eqnarray*} 
Let $\mathfrak{B}_{k+1}(0)$ be the Bergman kernel on the $z_n$-disk at $z'=0$ we studied in the last section. By lemma \ref{lem-2a}, we get that for $t>t_{\sqrt{k}\log k}$,
$$|\sum_{i\geq 2\sqrt{k}\log k}F_i(0)z_n^i|^2h_0^{k+1}(0,t)=\epsilon(k) \mathfrak{B}_{k+1}(0). $$ 	
Then we let $$P_2=\sum_{i< 2\sqrt{k}\log k}|F_i(0)|^2\int_{t<k^{1/3}}|z_n^i|^2h_0^k(0,t)\frac{|dz_n|^2}{|z_n|^2}=\sum_{i< 2\sqrt{k}\log k}|F_i(0)|^2J_i(0),$$
and $$Q_2=\int_{|z'|<\frac{\log k}{\sqrt{k}}}H^{-k}(z',0)\sum_{i< 2\sqrt{k}\log k}|F_i|^2\int_{t<k^{1/3}}|z_n^i|^2h_0^k\frac{|dz_n|^2}{|z_n|^2}.$$
By lemma \ref{lem-ja}, we have $$J_a=h_0^k(t_a)e^{-at_a}\sqrt{\pi k}\frac{2}{a}(1+O(\frac{\log k}{\sqrt{k}})),$$
for $a=O(\sqrt{k}\log k)$. For $t=O(\frac{\sqrt{k}}{\log k})$, we have
\begin{eqnarray*}
    \frac{h_0^k(z',t)}{h_0^k(0,t)}&=&1+O(k\frac{|z'|}{t^{1+\epsilon}})\\
    &=&1+O(\frac{(\log k)^3}{k^{\epsilon/2}}),
\end{eqnarray*} 
and also
$$at_a(z')-at_a(0)=O(\frac{a|z'|}{(2k/a)^{\epsilon}})=O(\frac{(\log k)^3}{k^{\epsilon/2}}).$$
Therefore $$\frac{J_a(z')}{J_a(0)}=1+O(\frac{(\log k)^3}{k^{\epsilon/2}}).$$
So \begin{eqnarray*}
    Q_2&=&(1+O(k^{-\epsilon}))\sum_{i< 2\sqrt{k}\log k}J_a(0)\int_{|z'|<\frac{\log k}{\sqrt{k}}}H^{-k}(z',0)||F_i(z')|^2|dz'|^2\\
    &\geq& (1+O(k^{-\epsilon}))\frac{\pi^{n-1}}{k^{n-1}}\sum_{i< 2\sqrt{k}\log k}J_a(0)|F_i(0)|^2\\
    &=& (1+O(k^{-\epsilon}))\frac{\pi^{n-1}}{k^{n-1}}P_2
\end{eqnarray*}
So we have 
\begin{eqnarray*}
    |\sum_{i< 2\sqrt{k}\log k}F_i(0)z_n^i|^2h_0^{k+1}(0,t)&\leq &P_2\mathfrak{B}_{k+1}(0)\\
    &\leq & (1+O(k^{-\epsilon}))\frac{k^{n-1}}{\pi^{n-1}}\mathfrak{B}_{k+1}(0)
\end{eqnarray*} 
So for $t>t_{\sqrt{k}\log k}$, $|s(0,z_n)|^2_{CY}\leq (1+O(k^{-\epsilon}))\frac{k^{n-1}}{\pi^{n-1}}\mathfrak{B}_{k+1}(0,z_n)$, where $\mathfrak{B}_{k+1}(0,z_n)$ means the value of $\mathfrak{B}_{k+1}(0)$ at $z_n$. The argument is of course valid for $t>t_{\lambda\sqrt{k}\log k}$ for a fixed $\lambda$. Since $s$ is arbitrary, we have proved

\begin{prop}\label{prop-general}
    For fixed $\lambda>0$, 
    $$\rho_{k+1}(0,z_n)\leq (1+O(k^{-\epsilon}))\frac{k^{n-1}}{\pi^{n-1}}\mathfrak{B}_{k+1}(0,z_n),$$
    for $t>t_{\lambda\sqrt{k}\log k}(0)$ and for some $\epsilon>0$.
\end{prop}
The same argument applies to $\rho_{k+1,a}$ and $\mathfrak{B}_{k+1,a}$, so we get 
\begin{prop}\label{prop-upper-a}
    For fixed $\lambda>0$, and fixed integer $a\geq 1$, we have
    $$\rho_{k+1,a}(0,z_n)\leq (1+O(k^{-\epsilon}))\frac{k^{n-1}}{\pi^{n-1}}\mathfrak{B}_{k+1,a}(0,z_n),$$
    for $t>t_{\lambda\sqrt{k}\log k}(0)$ and for some $\epsilon>0$.
\end{prop}
If we restrict $s$ to $\hcal_{k+1,a+1}$ for $a=O(\frac{\sqrt{k}}{\log k})$, then the same argument together with lemma \ref{lem-concentrated} give us
\begin{prop}\label{prop-hka}
    For $a=O(\frac{\sqrt{k}}{\log k})$, $$\rho_{k+1,a+1}(0,z_n)=  \epsilon(k)\mathfrak{B}_{k+1}(0,z_n),$$
    for $t>t_a(0)-\frac{k}{2a^2}$.
\end{prop}

By lemma \ref{lem-bk-between}, we have that for $c_1\frac{\sqrt{k}}{\log k}<t<c_2\sqrt{k}\log k$, $$|s(0,z_n)|^2_{CY}\leq \frac{k^{n-1}}{\pi^{n-1}}(4k+2\sqrt{k}t).$$
Therefore,
\begin{prop}\label{prop-rk-between}
    We have
    $$\rho_{k+1}(0,z_n)\leq \frac{k^{n-1}}{\pi^{n-1}}(4k+2\sqrt{k}t),$$
    if $c_1\frac{\sqrt{k}}{\log k}<t<c_2\sqrt{k}\log k$ for some fixed positive constants $c_1,c_2$.
\end{prop}
For $a=O(\frac{\sqrt{k}}{\log k})$,
we have
\begin{eqnarray*}
    \frac{h_0^k(z',t_a)}{h_0^k(0,t_a)}&=&1+O(k\frac{|z'|}{(2k/a)^{1+\epsilon}})\\
    &=&1+O(\frac{a^{1+\epsilon}(\log k)}{k^{1/2+\epsilon}}),
\end{eqnarray*} 
and also
$$at_a(z')-at_a(0)=O(\frac{a|z'|}{(2k/a)^{\epsilon}})=O(\frac{a^{1+\epsilon}(\log k)}{k^{1/2+\epsilon}}).$$
Therefore $$\frac{J_a(z')}{J_a(0)}=1+O(\frac{a^{1+\epsilon}(\log k)}{k^{1/2+\epsilon}}).$$
Now for $t>t_a-\frac{k}{2a^2}$, we can divide the summation in $P_2$ and $Q_2$ into two parts $\sum_{i\geq a+1}$ and $\sum_{i\leq 2} $. Then by lemma \ref{lem-concentrated}, we can ignore the part $\sum_{i\geq a+1}$. And so we get better precision 

\begin{prop}\label{prop-upper-bound}
    For $a=O(\frac{\sqrt{k}}{\log k})$, we have 
    $$\rho_{k+1}(0,z_n)\leq (1+O(\frac{a^{1+\epsilon}(\log k)}{k^{1/2+\epsilon}}))\frac{k^{n-1}}{\pi^{n-1}}\mathfrak{B}_{k+1}(0,z_n),$$
    for $t>t_a(0)-\frac{k}{2a^2}$ and for some $\epsilon>0$.
\end{prop}
\section{lower bound}\label{sec-lower}

For $a=O(\frac{\sqrt{k}}{\log k})$,
let $f=z_n^a$. We can estimate 
$$I(a)=\int_{U_k}H^{-k}(z',0)|z_n^a|^2h_0^k\frac{|dz|^2}{|z_n|^2},$$ as follows.
\begin{eqnarray*}
    I(a)&=&(1+O(\frac{a^{1+\epsilon}(\log k)}{k^{1/2+\epsilon}}))\int_{|z'|<\frac{\log k}{\sqrt{k}}} e^{-k|z'|^2}|dz'|^2\\
    &=&(1+O(\frac{a^{1+\epsilon}(\log k)}{k^{1/2+\epsilon}}))J_a(0)\frac{\pi^{n-1}}{k^{n-1}}
\end{eqnarray*}
More importantly, close to the border of $U_k$, $H^{-k}(z',0)|z_n^a|^2h_0^k$ is very small. More precisely, let $U_{k,1}=\{z\in U_k|t<3k^{1/3} \} $ and  $U_{k,2}=\{z\in U_k||z'|>\frac{\log k}{2\sqrt{k}} \} $, then we have
$$\int_{U_{k,j}}H^{-k}(z',0)|z_n^a|^2h_0^k\frac{|dz|^2}{|z_n|^2}=\epsilon(k)I(a),$$
for $j=1,2$.

We need to use H\"ormander's $L^2$ estimate to construct global sections from local holomorphic functions. The following lemma is well-known, see for example \cite{Tian1990On}. 

\begin{lem}
    Suppose $(M,g)$ is a complete \kahler manifold of complex dimension $n$, $\mathcal L$ is a line bundle on $M$ with hermitian metric $h$. If 
    $$\langle-2\pi i \Theta_h+Ric(g),v\wedge \bar{v}\rangle_g\geq C|v|^2_g$$
    for any tangent vector $v$ of type $(1,0)$ at any point of $M$, where $C>0$ is a constant and $\Theta_h$ is the curvature form of $h$. Then for any smooth $\mathcal L$-valued $(0,1)$-form $\alpha$ on $M$ with $\bar{\partial}\alpha=0$ and $\int_M|\alpha|^2dV_g$ finite, there exists a smooth $\mathcal L$-valued function $\beta$ on $M$ such that $\bar{\partial}\beta=\alpha$ and $$\int_M |\beta|^2dV_g\leq \frac{1}{C}\int_M|\alpha|^2dV_g$$
    where $dV_g$ is the volume form of $g$ and the norms are induced by $h$ and $g$.
\end{lem}

Fix $k$ large so that the assumption of the Lemma is satisfied in our setting with $M=X\setminus D$, the line bundle being $K_X^{k+1}$ equipped with the Cheng-Yau metric.
Let $\chi_1(y)$ be a smooth decreasing function satisfying the following
\begin{itemize}
    \item $\chi_1(y)=1$ for $y<\frac{(\log k)^2}{4k}$;
    \item  $\chi_1(y)=0$ for $y>\frac{(\log k)^2}{k}$;
    \item $0\leq \chi_1'(y)\leq \frac{2k}{(\log k)^2}$.
\end{itemize}
Let $\chi_2(y)$ be a smooth increasing function satisfying the following
\begin{itemize}
    \item $\chi_2(y)=1$ for $y>3k^{1/3}$;
    \item  $\chi_2(y)=0$ for $y<k^{1/3}$;
    \item $0<\chi_2'(y)\leq k^{-1/3} $.
\end{itemize}
Then we define the cut-off function $$\chi(z)=\chi_1(|z'|^2)\chi_2(t).$$
Then $$\dbar \chi=\chi_2(t)\chi_1'(|z'|^2)\sum z_i d\bar{z}_i-\chi_1|z'|^2\chi_2'(t)\frac{d\bar{z}_n}{\bar{z}_n}. $$
Then since $\omega_{KE}$ is equivalent to $\omega_{CG}$
we can estimate 
$$|\dbar \chi|^2_{CY}\leq C(\frac{k}{(\log k)^2}+t^2k^{-2/3}),$$
for some $C>0$.

Then $$\int_{U_k}H^{-k}(z',0)|\dbar \chi z_n^a|^2h_0^k\frac{|dz|^2}{|z_n|^2}\leq V_1+V_2+V_3,$$
where $$V_1=9C\int_{U_{k,1}}H^{-k}(z',0)|z_n^a|^2h_0^k\frac{|dz|^2}{|z_n|^2},$$ $$V_2=\frac{Ck}{(\log k)^2}\int_{U_{k,2}}H^{-k}(z',0)|z_n^a|^2h_0^k\frac{|dz|^2}{|z_n|^2},$$ and $$V_3=(\frac{Ck}{(\log k)^2}+9C)\int_{U_{k,1}\cap U_{k,2} }H^{-k}(z',0)|z_n^a|^2h_0^k\frac{|dz|^2}{|z_n|^2}.$$
Therefore, $$\int_{U_k}H^{-k}(z',0)|\dbar \chi z_n^a|^2h_0^k\frac{|dz|^2}{|z_n|^2}=\epsilon(k)I(a).$$
Considering $z_n^a\dbar\chi$ as a $K_X^{k+1}$-valued $(0,1)$-form,
we can solve the equation $$\dbar v=z_n^a\dbar\chi, $$
and get $v\in L^2(\xmd,K_X^{k+1},|\cdot|_{CY},\frac{\omega^n_{KE}}{n!})$ satisfying
\begin{eqnarray*}
    \int_{\xmd}|v|^2 _{CY}\frac{\omega^n_{KE}}{n!}&\leq&\frac{c_a}{k}\int_{U_k}H^{-k}(z',0)|\dbar \chi z_n^a|^2h_0^k\frac{|dz|^2}{|z_n|^2} \\
    &=&\epsilon(k)I(a)
\end{eqnarray*}
So we get a section $\chi z_n^a-v\in \hcal_{k+1}$ whose Cheng-Yau $L_2$-norm is $(1+\epsilon(k))I(a)$.
Then we have $$\parallel \chi z_n^a-v \parallel^2_{CY}=(1+\epsilon(k))I(a).$$
Since $v$ is holomorphic in $U_k\backslash U_{k,1}$ and 
$$\int_{U_k\backslash U_{k,1}}|v|_{CY}^2\frac{|dz|^2}{|z_n|^2}=\epsilon(k)I(a), $$
we get $$I(a)^{-1}|v(0,z_n)|_{CY}^2=\epsilon(k)\mathfrak{B}_{k+1}(0,z_n).$$
Since $$\frac{|z_n^a|^2}{I(a)}=(1+O(\frac{a^{1+\epsilon}(\log k)}{k^{1/2+\epsilon}}))\frac{|z_n^a|^2}{J_a(0)}\frac{\pi^{n-1}}{k^{n-1}},$$
by lemma \ref{lem-concentrated}, we get that
\begin{lem}
    For $a=O(\frac{\sqrt{k}}{\log k})$,
    $$\rho_{k+1}(0,z_n)\geq (1+O(\frac{a^{1+\epsilon}(\log k)}{k^{1/2+\epsilon}}))\frac{k^{n-1}}{\pi^{n-1}}\mathfrak{B}_{k+1}(0,z_n),$$
    for $|t-t_a(0)|\leq \frac{k}{2a^2}$.
\end{lem}
If we start with a function $f=\sum_{i=1}^{a}c_iz_n^i$, then for each $z'$, we know that for the integration $\int |f|^2h_0^k\frac{|dz_n|^2}{|z_n|^2} $, the part $t<t_a-\frac{k}{a^2}$ is $\epsilon(k)$ compared to the part $t\geq t_a-\frac{k}{a^2}$. And since for $t=O(\frac{2k}{a})$, we have $\frac{h_0^k(z',0)}{h_0^k(0,t)}=1+O(\frac{a^{1+\epsilon}|z'|}{k^\epsilon})$, we get that
$$\int_{U_k}H^{-k}(z',0)|f|^2h_0^k\frac{|dz|^2}{|z_n|^2}=(1+O(\frac{a^{1+\epsilon}\log k}{k^{1/2+\epsilon}}))\frac{\pi^{n-1}}{k^{n-1}}\int_{t>k^{1/3}}|f|^2h_0^k(0,t)\frac{|dz_n|^2}{|z_n|^2}.$$
Then we can repeat the process for $z_n^a$ and prove the following:
\begin{prop}\label{prop-lower-bound}
    For $a=O(\frac{\sqrt{k}}{\log k})$,
    $$\rho_{k+1}(0,z_n)\geq (1+O(\frac{a^{1+\epsilon}(\log k)}{k^{1/2+\epsilon}}))\frac{k^{n-1}}{\pi^{n-1}}\mathfrak{B}_{k+1}(0,z_n),$$
    for $t>t_a(0)- \frac{k}{2a^2}$ and for some $\epsilon>0$.
\end{prop}
In particular, for a fixed $a=O(1)$,
we have 
$$\rho_{k+1}(0,z_n)= (1+O(\frac{(\log k)}{k^{1/2+\epsilon}}))\frac{k^{n-1}}{\pi^{n-1}}\mathfrak{B}_{k+1}(0,z_n),$$
for $t>t_a(0)-\frac{k}{2a^2}$.
\begin{proof}[Proof of theorem \ref{thm-main-hka} ]
    For convenience, we work on $\hcal_{k+1}$.
    For any point $p_0$ satisfying the condition of the theorem, we let $p\in D$ be the point such that $d(p_0,p)=d(p_0,D)$. Then we can find a point $p'$ close to $p$ and a coordinates $z$ centered at $p'$ satisfying the conditions at the beginning of this section such that the coordinates of $p_0$ are $(0,w_n)$.  
    For any point $p_1$ with coordinates $z$, the distance $d(p_1,p')=|z|+O(|z|^2)$. Since  $d(p_0,p')\geq d(p_0,p)$, we have $|w_n|(1+O(|w_n|))\geq d(p_0,p)$. Let the coordinates of $p$ be $(x',0)$, then we have $|x'|(1+O(|x'|))\leq 2|w_n|(1+O(|w_n|))$. So 
    $$d(p_0,p)=\sqrt{|x'|^2+|w_n^2|}(1+O(|w_n|)).$$
    So $|w_n|^2(1+O(|w_n|))\geq |x'|^2+|w_n|^2$ and we get 
    $|x'|^2=O(|w_n|^3)$. Let $t=-\log |w_n|^2$, then since $\tau(p_0)>\frac{2k}{a+1}$, we have $|w_n|=\epsilon(k)$ and $$\tau(p_0)=t+\epsilon(k).$$
    Clearly, if we add a fixed constant to the bound $t_a(0)-\frac{k}{2a^2}$ in proposition \ref{prop-hka}, the statement will not change. Then by proposition \ref{prop-upper-bound}, proposition \ref{prop-lower-bound} and the argument for lemma \ref{lem-concentrated}, we have proved 
    the theorem.
    
\end{proof}
The same argument together with proposition \ref{prop-rk-between} gives us the proof of theorem \ref{thm-main-neck}.
\begin{proof}[Proof of theorem \ref{thm-main-inside}]
    For $a<c_1\frac{\sqrt{k}}{\log k}$, there is a constant $C_3$ depending only on $c_1$ such that $|t_a(p)-(\frac{2k}{a}-\frac{B}{2}\log \frac{2k}{a})|<C_3$ for $k$ large enough. Since $$\Lambda_a(t)h_a(t)=(1+O(\frac{1}{(\log k)^2}))\Lambda_a(t_a')h_a(t_a'),$$
    for $|t-t_a'|<C_3$, we get that 
    $$\rho_{k+1}(p_0)=(1+O(\frac{1}{(\log k)^2}))\frac{2k^{3/2}}{\sqrt{\pi}a},$$
	for $p_0\in \Sigma_a$ by proposition \ref{prop-upper-bound}, proposition \ref{prop-lower-bound} and lemma \ref{lem-inside}.
    And by the argument for lemma \ref{lem-concentrated}, we have
		$$\rho_{k+1}(p_0)=\epsilon(k),$$
	for $p_0\in \Sigma'_a$. And by proposition \ref{prop-upper-bound} and formula \ref{for-most-in}, we get the last statement of the theorem.
    

\end{proof}

\section{inner shells}
In this section, we deal with the case of fixed $a$, namely $a=O(1)$. We keep our choice of the local coordinates, namely the one that makes 
\begin{itemize}
    \item $\tij=\delta_{ij}+O(|z|)$;
    \item $\tij(z',0)=\delta_{ij}+O(|z'|^2)$  for $i<n$ and $j<n$;
    \item $\phi_1(z',0)=\phi_1(0,0)+O(|z'|^2).$
\end{itemize}
So we can find local frame $e_D$, such that when $e_L=e_D\otimes(dz_1\wedge\cdots\wedge dz_n)$, the local potential $\psi=|z|^2+2\Re \sum_{a\leq b} Q_{abc}z_az_b\bar{z}_c + O(|z|^4)$. Then it is easy to see that when $Q_{abc}\neq 0$, we have $b=n$ or $c=n$. Therefore, we have $$\psi(z',0)=|z'|^2+O(|z|^4).$$
Recall that $H(0,0)=1$, so by \ref*{eqn-H}  we have $G(0,0)=1$. But we use a new choice of the metric $h_D$ on $[D]$. Since $\omega^n$ defines a metric on $K_X$, we let $h_D$ be defined by $$|e_D|^2_{h_D}=\frac{e^{-\psi}\omega^n}{n!|dz|^2}.$$
Of course, this choice of $h_D$ will not guarantee $c_{1,0}$ being constant, but we will not need this in this section. 
In particular, in the local coordinates we just chosen near $D$, we have $\phi(0,0)=0$ and $\phi_1(0,0)=0$.
We define two open sets: the polydisc $\Omega_k\subset \C^n$ defined by $|z'|<\frac{\log k}{\sqrt{k}} \text{and} |z_n|<\frac{\log k}{\sqrt{k}}$; and the open set $W_k\subset \Omega_k$ defined by $t>qk$, where $q=\frac{1}{10a+2}$. By abuse of notation, we will also use $W_k$ to denote the open set in $X$ parametrized by $W_k$.
Let $p_0=(z',z_n)\in W_k$. Then for any $p_1=(w',0)$, we have $d(p_0,p_1)=\sqrt{|z'-w'|^2+|z_n|^2}(1+O(|z'|+|w')+|z_n|)$. Therefore $$d(p_0,D)=|z_n|(1+O(|z'|+|z_n|)).$$ 
So we have
\begin{equation}\label{tau-sigma}
    \tau(p_0)=t+O(\frac{\log k}{\sqrt{k}})=\sigma(p_0)+O(\frac{\log k}{\sqrt{k}}).
\end{equation}

\

To understand $\gcal_{k,a}$, we need to, for each $s\in H^0(D,kL-aD)$, estimate the minimal extension in $\hcal_k$. 
 We recall the following theorem from \cite{Finski}(Theorem 4.4).
\begin{theo}\label{thm-finski}
    There are $c>0$ and positive integer $p_1$ such that for any $p\geq p_1$ and $g\in H^0(D, L^p\otimes F) $, there is $f\in H^0(X,L^p\otimes F)$, such that $f|_D=g$ and 
    $$\parallel f\parallel_{L^2(X)}^2\leq \frac{c}{p}\parallel g\parallel_{L^2(D)}^2$$
    
\end{theo}
Let $F=a[D]$, with the induced hermitian metric from that of $[D]$. 
Given a section $s\in H^0(D,(k+1)L-a[D])$ with unit norm, we let $S\in H^0(X,(k+1)L-a[D])$ be the minimal extension of $s$. So we have
$$\int_X |S|^2 \omega^n\leq \frac{c_a}{k}\int_X |s|^2 \omega^{n-1},$$
where $c_a$ is a constant depending only on $a$. 
Then we define a section $\tilde{s}=S\otimes s_D^{\otimes a}\in H^0(X,(k+1)L)$. So by the identification of $L$ with $K_X$ on $X\backslash D$, we get a holomorphic pluri-canonical form $T(s)=\frac{\tilde{s}}{s_D^{\otimes (k+1)}}$ on $X\backslash D$.
Locally, recall that $e_L=e_D\otimes dz_1\wedge\cdots\wedge dz_n$, so if $S=F(e_L)^{\otimes (k+1)}\otimes (e_D^*)^{\otimes a}$, where $e_D^*$ is the dual frame of $e_D$, then $\tilde{s}$ is represented by $Fz_n^aG^a(e_L)^{\otimes (k+1)}$. So $T(s)=\frac{f}{(z_nG)^{k+1-a}}(dz_1\wedge\cdots\wedge dz_n)^{\otimes (k+1)}$. 
So its pointwise Cheng-Yau norm is 
$$|T(s)|^2_{CY}=\frac{|Fz_n^a|^2}{|G^{k+1-a}|^2}h_{CY}^{k+1}.$$
and $$|T(s)|^2_{CY}\frac{\omega_{KE}^n}{n!|dz|^2}=\frac{|Fz_n^{a-1}|^2}{|G^{k+1-a}|^2}h_{CY}^{k}.$$
Recall that $H(z',0)=\frac{e^{\psi(z',0)}}{|G(z',0)|^2}$.
So $$|T(s)|^2_{CY}\frac{\omega_{KE}^n}{n!|dz|^2}=|F|^2|Gz_n|^{2a-2}(\frac{|G(z',0)|^2}{|G(z)|^2})^{k}(Yt^2)^k e^{-k\psi(z',0)}$$

Clearly we have $\frac{G(z',z_n)}{G(z',0)}=1+O(|z_n|)$. Then by the compactness of $D$, we can find universal constants $R$ and $C_1$ so that $B_R(p)$, the geodesic ball of radius $R$ under the Riemannian metric associated to $\omega$, is contained in the coordinates chart centered at $p$ and that $|\frac{G(z',z_n)}{G(z',0)}-1|<C_1|z_n|$ within $B_R(p)$. Then on $U_k$, we have 
$$|T(s)|^2_{CY}\frac{\omega_{KE}^n}{n!|dz|^2}=(1+o(\frac{1}{k^2}))|F|^2|Gz_n|^{2a-2}h_0^k(z',t) e^{-k\psi(z',0)}.$$

We will need the following basic technical lemma.
\begin{lem}\label{lem-concave}
    Let $f(x)$ be a concave function. Suppose $f'(x_0)<0$, then we have
    $$\int_{x_0}^\infty e^{f(x)}dx\leq\frac{e^{f(x_0)}}{-f'(x_0)}$$
\end{lem}

\

Write $F(z)=\sum F_i(z')z_n^i$ as a power series of $z_n$.  We denote by 
$$Q(z')=\int_{|z_n|<\frac{\log k}{\sqrt{k}}}|F(z',z_n)|^2e^{-(k+1)\psi+a\phi_1}\frac{\omega^n}{n!|dz'|^2}.$$

Since $\frac{\omega^n}{n!}=(1+O(|z|))|dz|^n$, we have $$Q(z')>(1-\nu)e^{-(k+1)|z'|^2+a\phi_1(0,0)}\int_{|z_n|<\frac{\log k}{\sqrt{k}}}|F(z',z_n)|^2e^{-k|z_n|^2}\sqrt{-1}dz_n\wedge d\bar{z}_n,$$
where $\nu$ is a quantity of the size $O(\frac{(\log k)^3}{\sqrt{k}})$.
And $$\int_{|z_n|<\frac{\log k}{\sqrt{k}}}|F(z)|^2e^{-k|z_n|^2}\sqrt{-1}dz_1\wedge d\bar{z}_1=2\pi \sum |F_i|^2\int_{r<\frac{\log k}{\sqrt{k}}} 2e^{-kr^2}r^{2i+1}dr$$
$w_i=\int_{r<\frac{\log k}{\sqrt{k}}} 2e^{-kr^2}r^{2i+1}dr=\int_0^\frac{(\log k)^2}{k} e^{-kx}x^{i}dx$. So we have 

\begin{eqnarray*}
    2\pi \sum w_id_i&<&2\int_{\Omega_k}|F|^2e^{-(k+1)\psi+a\phi}|dz|^2\\
    &<&\frac{2c_a}{k+1},
\end{eqnarray*}
where $d_i=e^{a\phi(0,0)}\int_{|z'|<\frac{\log k}{\sqrt{k}}}|F_i|^2e^{-(k+1)|z'|^2}|dz'|^2$. Then we estimate $w_i$. Letting $y=kx$, we get $w_i=\frac{1}{k^{i+1}}\int_0^{(\log k)^2}y^ie^{-y}dy$. Since $-y+i\log y$ is a concave function of $y$, whose only maximum is attained at $y=i$, one can see that for $i<(\log k)^2-\log k$, we have $\frac{1}{2}<\frac{k^{i+1}w_i}{i!}<\frac32$. For $i\geq (\log k)^2-\log k$, we have that $y^ie^{-y}$ is increasing on the interval $[0,(\log k)^2-\log k]$. So we have 
$w_i\geq \frac{1}{k^{i+1}}((\log k)^2-2\log k)^ie^{-(\log k)^2+2\log k}$.

\
Then we consider the integral of $|T(s)|^2_{CY}\frac{\omega_{KE}^n}{n!|dz|^2}$ along the $z_n$-disks with euclidean measure $\sqrt{-1}dz_n\wedge d\bar{z}_n$.

Let $U(z')=\int_{\log |z_n|^2<\phi_1(z',0)-qk} |F|^2r^{2a-2}h_0^k(t)\sqrt{-1}dz_n\wedge d\bar{z}_n$, then
\begin{eqnarray*}
    U(z')&=&2\pi \sum |F_i|^2 \int_{\log |z_n|^2<\phi(z')-qk}h_0^k(t)|r|^{2i+2a-2} 2rdr\\
    &=&2\pi \sum |F_i|^2 \int_{qk-\phi(z',0)}^\infty h_0^k(t)e^{-(i+a)t} dt
\end{eqnarray*}
Let
$y_i(z')=\int_{qk-\phi(z',0)}^\infty h_0^k(z',t)e^{-(i+a)t} dt$.

Clearly $y_0(z')=(1+\epsilon(k))J_a(z')$, so we have that $$y_0(z')=(1+O(k^{-\epsilon}))(\frac{2k}{a})^{2k+Ba/2}\frac{2}{a}\sqrt{k\pi }e^{-2k+Ba-Aa/2},$$
and for $i< 20a+2$, $$\frac{y_i}{y_0}(z')<2(\frac{a}{a+i})^{2k+1+Ba/2}e^{Bi-Ai/2}(a+i)^{-Bi/2}.$$ For $i\geq 20a+2$, we can use lemma \ref{lem-concave} to get 
$$y_i(z')\leq \frac{2}{i-10a}(qk)^{2k}e^{-\frac{i+a}{10a+4}k}.$$
In the next lemma and its proof, for simplicity of the notations, we ignore the dependence of $y_i$ on $z'$, since the term $A$ does not change much for $|z'|<\sqrt{k}\log k$.
\begin{lem}
    We have $\frac{y_iw_0}{y_0w_i}\leq \epsilon(k)$, where the quantity $\epsilon(k)$ can be chosen independent of $i$.
\end{lem}
\begin{proof}
    For $i<20a+2$, we have 
    $$\frac{y_iw_0}{y_0w_i}< \frac{4k^i}{i!}(\frac{a}{a+i})^{2k+1+Ba/2}e^{Bi-Ai/2}(a+i)^{-Bi/2}=\epsilon(k).$$
    For $(\log k)^2-\log k>i\geq 20a+2$, we have
    $$\frac{y_iw_0}{y_0w_i}< \frac{2a^{Ba/2+1}k^i}{(i-10a)i!\sqrt{k\pi}}(\frac{a}{20a+4})^{2k}e^{2k-Ba+Aa/2-\frac{i+a}{10a+4}k}.$$
    We look at $i\log k-\log( i!)-2k\log 20+2k-\frac{i+a}{10a+4}k$. It is then easy to see that for $k$ large enough, $\frac{y_iw_0}{y_0w_i}=\epsilon(k)$. 
    For $i\geq (\log k)^2-\log k$, we have
    $$\frac{y_iw_0}{y_0w_i}< \frac{k^{i}}{(20)^{2k}((\log k)^2-2\log k)^i}e^{2k-\frac{i+a}{10a+4}k+(\log k)^2-2\log k}.$$
    Again, by taking logarithm again, it is easy to see that $\frac{y_iw_0}{y_0w_i}\leq \epsilon(k)$ with a quantity $\epsilon(k)$ independent of $i$. 
\end{proof}
Therefore, $$\sum_{i\geq 1}y_i(0)d_i=\epsilon(k)y_0(0).$$ 

\

In the following, unless stated otherwise, we will assume that $s$ is supported in one component $D_\alpha$. And by $B$ we will mean $B_\alpha$.

So we have proved the following:
\begin{lem}\label{lem-wk}
    Let $\nu(p)=e^{-a\phi_1(0,0)}(\frac{2k}{a})^{2k+Ba/2}\frac{2}{a}\sqrt{k\pi }e^{-2k+Ba-A(0)a/2},$ then
    $$\int_{W_k}|\sum_{i=1}^{\infty}F_iz_n^i|^2|Gz_n|^{2a-2}h_0^k(z',t) e^{-k\psi(z',0)}=\epsilon(k)\nu(p).$$
    So
    $$\int_{W_k}|T(s)|^2_{CY}\frac{\omega_{KE}^n}{n!}=(1+O(k^{-\epsilon}))\nu(p)[\int_{z'<\sqrt{k}\log k}|s|^2\frac{\omega^n}{n!}+\epsilon(k)].$$
\end{lem}
\begin{proof}
    The reader only need to be reminded that since $G(0,0)=1$, $\phi(0,0)=\phi_1(0,0)$.
\end{proof}
When we vary $p$ in $D$, $A$ and $\phi_1(p)=\phi_1(0,0)$ are uniformly bounded. Therefore, for any $p_1,p_2\in D$, $t_a(p_1)-t_a(p_2)=O(1)$.
For $t>\sqrt{k}\log k$, we have $\frac{h_0(t_a(p_1))}{h_0(t_a(p_2))}=1+O((\frac{2k}{a})^{1+\epsilon})$. So 
$$\frac{J_a(p_1)}{J_a(p_2)}=e^{O(a)}.$$
We fix a neighborhood $N_\alpha$ for each $D_\alpha$ and we require that $N_\alpha\cap N_\beta=\empty$ when $\alpha\neq \beta$. Let $s$ be supported in $D_\alpha$. 
Then we have 
\begin{prop}\label{prop-tsqk}
    There are positive constants $c_1(a), c_2(a)$ such that the quotient
    $$\xi=(\int_{\{\sigma>qk\}\cap N_\alpha}|T(s)|^2_{CY}\frac{\omega_{KE}^n}{n!})/((\frac{2k}{a})^{2k+Ba/2}\frac{2}{a}\sqrt{k\pi }e^{-2k})$$
    satisfies 
    $$c_1(a)<\xi<c_2(a).$$
\end{prop}

To control the norm of $T(s)$ on the points further away from $D$, we use H\"{o}rmander's $L_2$ estimates again. Firstly, we notice:
\begin{lem}
    $$\int_{0<\sigma-qk<2}|T(s)|^2_{CY}\frac{\omega_{KE}^{n}}{n!}=\epsilon(k)(\frac{2k}{a})^{2k+Ba/2}e^{-2k}.$$
\end{lem}
\begin{proof}
    One only need to notice that on each $z_n$ the disk, the integral of $h_0^k(t)e^{-at}$ on the part where $0<\sigma-\frac{k}{10a+2}<2$ is $\epsilon(k)y_0$. 
\end{proof}

Let $\chi(y)$ be a smooth function satisfying the following
\begin{itemize}
    \item $\chi(y)=1$ for $y>2$;
    \item  $\chi(y)=0$ for $y<0$;
    \item $0\leq \chi'(y)\leq 1$ for $y>2$.
\end{itemize}

We define a function $\zeta$ such that $\zeta=\chi(\sigma-qk)$ on $N_\alpha$ and $\zeta=0$ on the complement of $N_\alpha$. Then $\zeta$ is smooth when $k$ is large enough.
Then the function $\zeta$ satisfies
$\dbar \zeta=\chi'(\sigma-qk)\dbar \sigma$. So 
$$|\dbar \zeta|^2\leq |\dbar \sigma|^2.$$
Then since $\omega_{KE}$ is equivalent to $\omega_{CG}$, we have
$|\dbar \sigma|^2=O(\sigma^2)$. Therefore,
$$\dbar \zeta=O(\sigma^2).$$

So we can solve the equation $$\dbar v=\dbar \zeta\otimes T(s) $$
and get 
\begin{eqnarray*}
    \int_X |v|^2\omega_{KE}^n&\leq& \frac{2}{k}\int_{\sigma>qk}|\dbar \zeta|^2|T(s)|_{CY}^2\omega_{KE}^n\\
    &\leq&\frac{2}{k}\int_{qk+2>\sigma>qk}\sigma^2|T(s)|_{CY}^2\omega_{KE}^n\\
    &\leq&\frac{2(qk+2)^2}{k}\int_{qk+2>\sigma>qk}|T(s)|_{CY}^2\omega_{KE}^n\\
    &=&\epsilon(k)(\frac{2k}{a})^{2k+Ba/2}e^{-2k}.
\end{eqnarray*}

So $E(s)=\zeta T(s)-v$ is a holomorphic section of $K_X^{k+1}$ on $\xmd$ satisfying
$$\parallel E(s)\parallel^2_{CY}=(1+\epsilon(k))\int_{\{\sigma>qk\}\cap N_\alpha}|T(s)|^2_{CY}\frac{\omega_{KE}^n}{n!}.$$ And we can require $v$ to be the minimal solution, namely $v$ is orthogonal to the holomorphic sections. With this requirement, one sees that $E(s)$ is a linear map.

We have the following
\begin{lem}\label{lem-es}
    Let $V_{\alpha,a}$ denote the open set $\{\sigma>\frac{2k}{a}+\frac{\sqrt{k}\log k}{a}\}\cup\cup_{\beta\neq \alpha}N_\beta$.
    We have
    \begin{equation}\label{eqn-Es-u}
        \int_{V_{\alpha,a}}|E(s)|^2\omega_{KE}^n=\epsilon(k)\parallel E(s)\parallel^2_{CY}
    \end{equation}
    \begin{equation}\label{eqn-Es-l}
        \int_{\sigma<\frac{2k}{a}-\frac{\sqrt{k}\log k}{a}}|E(s)|^2\omega_{KE}^n=\epsilon(k)\parallel E(s)\parallel^2_{CY}
    \end{equation}
    
\end{lem}
\begin{proof}
    We only need to look at the section $\zeta T(s)$. Locally, by lemma \ref{lem-wk}, we only need to look at $F_0z_n^a$. Then the conclusion by the mass concentration property of $J(a)$ in $W_k$. Then by covering a neighborhood of $D$ with open sets of the form $W_k$, we get the conclusion.
\end{proof}
As a direct corollary of the lemma, the inner product 
$$\langle E(s_\alpha),E(s_\beta) \rangle=\epsilon(k)\parallel E(s_\alpha) \parallel \parallel E(s_\beta) \parallel,$$
for $\alpha\neq \beta$. 

Now we drop the assumption on the support of $s$ and consider general $s\in H^0(D,(k+1)L-aD)$. 

\

Let $s=\sum_{\alpha\in \Lambda} s_\alpha$ such that the support of $s_\alpha$ is in $D_\alpha$. Then 
we define 
$$E(s)=\sum_{\alpha\in \Lambda} E(s_\alpha)$$ 
We denote by $\pi_a$ the isomorphism $\gcal_{k+1,a}\to H^0(D,(k+1)L-aD)$ defined by $s\mapsto (\frac{s\otimes s_D^{\otimes (k+1)}}{s_D^a})|_D$. So the inverse $\pi_a^{-1}(s)$ can be considered as minimal "extension" of $s$. Then we have the following

\begin{prop}
    We have
    $$\parallel E(s)-\pi_a^{-1}(s) \parallel^2_{CY}=\epsilon(k)\parallel E(s)\parallel^2_{CY},$$	
    and $$\int_{\sigma<\frac{2k}{a}-\frac{\sqrt{k}\log k}{a}}|\pi_a^{-1}(s)|^2_{CY}\omega_{KE}^n=\epsilon(k)\parallel E(s)\parallel^2_{CY}.$$
\end{prop}
\begin{proof}
    We first notice that it suffices to show that the conclusions hold when $s=s_\alpha$ for some $\alpha$. And we can assume that $\parallel s\parallel=1$. Then we do induction on $a$. When $a=1$, we can decompose $E(s)=\theta_2+\theta_3$, where $\theta_2\in \gcal_{k+1,1}$, $\theta_3\in \hcal_{k+1,2}$. By theorem \ref{thm-main-hka} and formula \ref{eqn-Es-l}, we have $\parallel \theta_3\parallel^2_{CY}=\epsilon(k)\parallel E(s)\parallel^2_{CY}$. 

    let $(\pi_a(\theta_{2}))_\alpha=s'\in H^0(D,(k+1)L-aD)$, then clearly, $s'$ satisfies the conclusion of the current proposition.

    Since $\dbar \chi(\sigma-qk)=0$ when $\sigma>qk+2$, the solution $v$ is holomorphic on the region where $\sigma>qk+2$. Let $p\in D_\alpha$ and let $W_k$ be the corresponding open set centered at $p$. 
    Then on $W_k\cap \{\sigma>qk+2\}$, if we write $v=(\sum v_iz_n^i)e_K^{k+1}$ as power series of $z_n$, then we have 
    \begin{equation}\label{eqn-va}
        J_a(0)\int_{|z'|<\sqrt{k}\log k}|v_a|^2 e^{-(k+1)\psi(z',0)+a\phi(z',0)}\frac{\omega_D^{n-1}}{(n-1)!}=\epsilon(k)\parallel E(s)\parallel^2_{CY}.
    \end{equation}
    Therefore, since $s-s'=v_a$ on $W_k$ and the integral in equation \ref{eqn-va} calculate the norm of $v_a$ on $W_k$, by proposition \ref{prop-tsqk}, we have 
    \begin{equation}\label{eqn-s'}
    \parallel s-s'\parallel^2\leq \epsilon(k),
    \end{equation}
    
    
    It is not hard to see that the right hand side of formula \ref{eqn-s'} can be made independent of $s$, which we denote by $o_{k+1,a}$. So we have the following statement: for any unit $s$, $\exists s_1$ which is also a unit section and satsifies the current proposition such that $$\parallel s-s_1\parallel\leq 2\sqrt{o_{k+1,a}}.$$
    So we can choose almost orthonormal basis $(s_1,\cdots,s_m)$, where $m=O(k^{n-1})$, such that all $s_i$ are unit sections and satsifies the current proposition and $$|\langle s_i,s_j\rangle|\leq 4\sqrt{o_{k+1,a}}. $$ 
    Then by expressing $s=\sum c_is_i$, we get that the proposition holds for $s$.

    When $a\geq 2$, assume that the proposition has been proved for $i\leq a-1$. Then we can decompose $E(s)=\theta_1+\theta_2+\theta_3$, where $\theta_1\in \bigoplus_{i=1}^{a-1}\gcal_{k+1,i} $, $\theta_2\in \gcal_{k+1,a}$, $\theta_3\in \hcal_{k+1,a+1}$. By theorem \ref{thm-main-hka} and formula \ref{eqn-Es-l}, we have $$\parallel \theta_3\parallel^2_{CY}=\epsilon(k)\parallel E(s)\parallel^2_{CY}.$$
    We can write $\theta_1=\sum_{i=1}^{a-1}\pi_b^{-1}f_b$, then 
    by induction each $f_b$ satisfies the current proposition and then by formula \ref{eqn-Es-u}, we must have
    $\parallel \theta_1\parallel^2_{CY}=\epsilon(k)\parallel E(s)\parallel^2_{CY}$.
    
    Similar to the estimates for $v$, if $\theta_1$ is represented by the holomorphic function $\sum g_iz_n^i$, we also have 
    $$J_a(0)\int_{|z'|<\sqrt{k}\log k}|g_a|^2 e^{-(k+1)\psi(z',0)+a\phi(z',0)}\frac{\omega_D^{n-1}}{(n-1)!}=\epsilon(k)\parallel E(s)\parallel^2_{CY}.$$
    Therefore, let $s'=(\pi_a(\theta_{2}))_\alpha$, then 
    \begin{equation}\label{eqn-s''}
    \parallel s-s'\parallel^2\leq \epsilon(k),
    \end{equation}
    Then the rest of the proof is same as the proof for the case $a=1$, and we have proved the proposition.

\end{proof}
As a corollary, we get
\begin{cor}There are positive constants $c_1(a), c_2(a)$ such that for any unit section of $ H^0(D,(k+1)L-aD)$ that is supported in some $D_\alpha$, the quotient
$$\xi=\frac{\parallel E(s)-\pi_a^{-1}(s) \parallel^2_{CY}}{(\frac{2k}{a})^{2k+B_\alpha a/2}\frac{2}{a}\sqrt{k\pi }e^{-2k}}$$
satisfies 
$$c_1(a)<\xi<c_2(a).$$
\end{cor}
We get that for any $S\in \gcal_{k+1,a}$ of unit norm, 
$$\int_{\sigma<\frac{2k}{a}-\frac{\sqrt{k}\log k}{a}}|S|_{CY}^2\omega_{KE}^n=\epsilon(k).$$
Then we can go back to local open sets $W_k$ and repeat the proof of proposition \ref{prop-upper-bound} to see that $|S|_{CY}^2=\epsilon(k)$ for the points where $2k^{1/3}<\sigma<\frac{2k}{a}-\frac{\sqrt{k}\log k}{a}$, hence for the points where $2k^{1/3}<\tau<\frac{2k}{a}-\frac{\sqrt{k}\log k}{a}$ by formula \ref{tau-sigma}. For the points where $\tau\leq 2k^{1/3}$, we can use the peak sections. Let $s_p$ be the peak section at such a point. It is then easy to see that the inner product of $S$ with $s_p$ is $\epsilon(k)$. So $|S(p)|_{CY}^2=\epsilon(k)$. 
In conclusion, $|S(p)|_{CY}^2=\epsilon(k)$ for all the points $\tau<\frac{2k}{a}-\frac{\sqrt{k}\log k}{a}$.
Since the dimension of $ H^0(D,(k+1)L-aD)$ is $O(k^{n-1})$, we get  theorem \ref{thm-main-vka}.

\

\bibliographystyle{plain}

\bibliography{references}

\end{document}